\documentclass[11pt]{amsart}

\usepackage{amssymb, array, verbatim, amscd, color}
\usepackage{amsmath, amsfonts,amsthm}
\usepackage{enumerate}
\usepackage[all]{xy}
\usepackage{xy}

\theoremstyle{plain}

\newtheorem{theorem}{Theorem}[section]
\newtheorem{corollary}[theorem]{Corollary}
\newtheorem{lemma}[theorem]{Lemma}

\theoremstyle{definition}

\newtheorem{example}[theorem]{Example}

\newtheorem{claim}{\it Claim}

\newcommand{\bbQ}{\mathbb Q}

\newcommand{\bbP}{\mathbb P}

\newcommand{\N}{{\mathbb N}}

\newcommand{\Q}{{\mathbf Q}}
\newcommand{\Z}{{\mathbb Z}}

\newcommand{\Ker}{\mathrm{Ker}}

\newcommand{\End}{\operatorname{End}}
\newcommand{\Hom}{\operatorname{Hom}}

\newcommand{\type}{\operatorname{\mathbf{type}}}

\newcommand{\TF}{\mathcal{TF}}

\begin{document}

\title{Self-pure-generators over Dedekind domains} 
\author{Simion Breaz} 
\address{Babe\c s-Bolyai University, Faculty of Mathematics
and Computer Science, Str. Mihail Kog\u alniceanu 1, 400084
Cluj-Napoca, Romania}
\email{bodo@math.ubbcluj.ro}




\begin{abstract} 
We prove that all pure submodules of a finite rank torsion-free module $A$ over a Dedekind domain are $A$-generated (i.e. $A$ is a self-pure-generator) if and only if $A$ has a rank $1$ direct summand $B$ such that $\mathbf{type}(B)$ is the inner type of $A$. This result is applied to describe the direct products of torsion-free groups of finite rank which are self-pure-generators. 
\end{abstract}


\maketitle


\section{Introduction}
Let $R$ be a unital ring. We will say that an $R$-module $A$ is a \textsl{self-pure-generator} if every pure submodule $K$ of $A$ is $A$-generated, i.e. there exists an epimorphism $A^{(I)}\to K$. The study of self-pure-generators was initiated in \cite{cal-et-al}, where the authors proposed the study of self-pure-generators abelian groups. A characterization for self-pure-generators which are direct products of rank $1$ torsion-free abelian groups is presented in \cite[Proposition 2]{cal-et-al}. 

The investigation of self-pure-generators is a natural extension of the study of self-generator modules (a module  $M$ is a self-generator if every submodule $K$ of $M$ is $M$-generated, \cite{Zimm}). Self-generator modules play an important role in the study of module categories since in many cases there are strong connections between the properties of a self-generator and the properties of its endomorphism ring (e.g. \cite[Lemma 2.4]{Fuller-density}, \cite[Lemma 1.4]{Zimm}). However, in many cases the self-generator condition is very restrictive: a torsion-free module over a Dedekind domain $R$ is a self-generator if and only if it has a direct summand isomorphic to $R$, \cite[Theorem 2.1]{Zimm}. Fortunatelly, in some cases it is enough to ask that some special submodules of a module $A$ are $A$-generated. For instance, an abelian group $A$ is an $\mathcal{S}$-group if every finite index subgroup of $A$ is $A$-generated, \cite{Ar-S-group}, and $A$ is flat as a left module over its endomorphism ring if and only if for every homomorphism $\alpha:A^n\to A$ the kernel of $\alpha$ is $A$-generated, \cite[Lemma 1.3]{Zimm}. Albrecht and Goeters describe some connections between these two classes of groups in \cite{Al-Go}, where they use a class of finite rank torsion-free groups $A$ such that all pure subgroups of rank $1$ are $A$-generated, \cite[Lemma 2.1]{Al-Go}. We refer to \cite{alb-cont} for other properties connected with $A$-generated subgroups of finite powers of $A$.   

In Theorem \ref{tffr-main}, we characterize the torsion-free modules of finite rank over Dedekind domains which are self-pure-generators as those finite rank torsion-free modules $A$ which have a direct decomposition $A=B\oplus C$ such that $B$ of rank $1$ and it generates all rank $1$ pure submodules of $C$. Using this result we obtain in Theorem \ref{direct-products} a similar charaterization for direct products of finite rank torsion-free abelian groups which are self-pure-generators. This completes \cite[Proposition 2]{cal-et-al}, where direct products of rank $1$ torsion-free abelian groups are considered. In the end we use the groups constructed in Example \ref{example} to observe that the above results are not valid for infinite direct sums of torsion-free abelian groups of finite rank, i.e. there exists a self-pure generator torsion-free abelian group $G$ which is a direct sum of an infinite family of rank $2$ abelian groups such that $G$ has no direct summands of rank $1$.  

We refer to \cite{lady} and \cite{lady-ja} for basic notions and techniques used in the study of finite rank torsion-free modules over Dedekind domains. Some of these notions and results are borrowed from the theory of finite rank torsion-free abelian groups, \cite{Arnold}. We also use \cite{Fuchs} for specific results about torsion-free abelian groups.

\section{Finite rank torsion-free modules}

Let $R$ be a Dedekind domain. We denote by $\Q$ its field of quotients. A torsion-free $R$-module $A$ is of rank $n\in\N$ if it can be embedded as an essential $R$-submodule of $\Q^n$. If $A$ is a finite rank torsion-free $R$-module, and $\End(A)$ is the endomorphism ring of $A$ then the ring $\Q\End(A)=\Q\otimes_R \End(A)$ is the ring of quasi-endomorphisms of $A$. This is an artinian $\Q$-algebra, its Jacobson radical $J$ is nilpotent and $J=\Q\otimes_R N$ (cf. \cite[Proposition 9.1]{Arnold}), where $N$ is the nil-radical of $\End(A)$. We recall that $\Q\End(A)$ determines the quasi-decompositions of $A$, i.e. the decompositions of $A$ in the additive category $\Q\TF$ whose objects are finite rank torsion-free $R$-modules, and $\Hom_{\Q\TF}(A,B)=\Q\otimes_R \Hom(A,B)$. The homomorphisms in $\Q\TF$ are called \textsl{quasi-homomorphisms}. Isomomorphisms in $\Q\TF$ are called \textsl{quasi-isomorphisms}. A quasi-homomorphism between $A$ and $B$ can be identified with a $\Q$-linear map $f$ between the injective envelopes $\Q A$ and $\Q B$ of $A$ and $B$ with the property that there exists an element  $r\in R^*$ such that $rf(A)\subseteq B$. In this context it is useful to remind that two $R$-modules $A$ and $C$ are \textsl{quasi-equal} if $\Q A=\Q C$ and there exist $r,s\in R^*$ such that $rC\leq A$ and $sA\leq C$. Therefore, a homomorphism $f:A\to B$ represents an isomorphism in $\Q\TF$ if and only if it is injective, and $f(A)$ is quasi-equal to $B$.        



A decomposition of a finite rank torsion-free $R$-module $A$ as a direct sum $B\oplus C$ in the category $\Q\TF$  is called a \textsl{quasi-decomposition}. This means that the $R$-module $B\oplus C$ is isomorphic to a submodule of $\Q A$ which is quasi-equal with $A$. Since the multiplication by $r\in R^*$ is an isomorphism in $\Q\TF$, it follows that we can identify the quasi-decompositions of a finite rank torsion-free $R$-module $A$ with the direct decompositions of those submodules of $A$ which are quasi-equal to $A$. An $R$-module which is indecomposable in $\Q\TF$ is called \textsl{strongly indecomposable}. Since the quasi-endomorphism rings of finite rank torsion-free $R$-modules are finite dimensional $\Q$-algebras, an $R$-module $A$ is strongly indecomposable if and only if $\Q\End(A)$ is a local ring. Therefore, an endomorphism of a strongly indecomposable $R$-module is either nilpotent or a quasi-isomorphism. It follows that the category $\Q\TF$ is a Krull-Schmidt category, i.e. the quasi-decompositions into direct sums of strongly indecomposable $R$-modules are unique up to quasi-isomorphism.    

We recall that every torsion-free $R$-module of rank $1$ can be identified, up to an isomorphism, with a submodule of 
$\Q$. The \textsl{type} of a torsion-free $R$-module $B$ of rank $1$ is denoted by $\type(B)$ and it represents the quasi-isomorphism class of $B$. Two non-zero submodules $A$ and $B$ of $\Q$ have the same type if and only if there exists a fractional ideal $I$ such that $A=IB$ (hence $A$ is $B$-generated). The set of types is a lattice with respect the partial order relation defined by $\type(A)\leq \type(B)$ if and only if there exists a monomorphism $A\to B$. If $A$ and $B$ are non-zero submodules of $Q$ then $\inf\{\type(A),\type(B)\}=\type(A\cap B)$ and $\sup\{\type(A),\type(B)\}=\type(A+B)$. Let us point out that in the case of abelian groups two rank $1$ torsion-free abelian groups of the same type are isomorphic, and in this case the types can be described by using some numerical invariants. We refer to \cite{Arnold} and \cite{Fuchs} for the classification theory of torsion-free groups of rank $1$. 

Let $A$ be a torsion-free $R$-module of finite rank. If $X\subseteq A$, we will denote by $\langle X\rangle_\star$ the pure submodule of $A$ which is generated by $X$. If $a\in A$ is non-zero then the type of $a$ is defined as $\mathbf{type}( a)=\mathbf{type}(\langle a\rangle_\star)$. If $x_1,\dots,x_n$ is a maximal independent system of $A$ then \textsl{the inner type} of $A$ is computed as $$\mathbf{IT}(A)=\inf\{\type(\langle x_1\rangle_\star), \dots,\type(\langle x_n\rangle_\star)\}.$$ Let us recall that $\mathbf{IT}(A)$ is independent of the choice of the maximal independent system $x_1,\dots,x_n$, and for every $a\in A$ we have $\mathbf{IT}(A)\leq \mathbf{type}( a)$ (in fact  
$\mathbf{IT}(A)$ is the greatest lower bound of the types of all non-zero elements of $A$). 

We recall some basic properties of finite rank torsion-free modules over Dedekind domains. 
The next lemma is proved in \cite[Proposition 4.3]{lady}. For reader's convenience we present a short outline of the proof.

\begin{lemma}\label{lemma1} Let $B$ be a rank $1$ torsion-free $R$-module.
A torsion-free $R$-module finite rank $A$ is $B$-generated if and only if $\mathbf{type}(B)\leq \mathbf{IT}(A)$.
\end{lemma}

\begin{proof} The direct implication is obvious since for every non-zero homomorphism $f:B\to A$ we have $\Ker(f)\neq B$ and $\Ker(f)$ is a pure submodule of $B$. Then $f$ has to be a monomorphism.

Conversely, suppose that $\mathbf{type}(B)\leq \mathbf{IT}(A)$. It is enough to prove that for every non-zero element $a\in A$ the rank $1$ pure-submodule $C=\langle a\rangle_\star$ is $B$-generated. The module $C$ is $B$-generated if and only if the natural map $$\varphi:\Hom(B,C)\otimes_R B\to C,\  \varphi(f\otimes a)=f(a)$$ is an epimorphism. We note that $\Hom(B,C)\neq 0$ since $\mathbf{type}(B)\leq\mathbf{type}(C)$. Since $R$ is Dedekind, all prime ideals are maximal, hence $\varphi$ is an epimorphism if and only if for all prime ideals $\mathbf{p}$ the localisation $\varphi_\mathbf{p}$ is an epimorphism. 

Let $\mathbf{p}$ be a prime ideal. We can suppose that $B\leq C\leq \Q$, hence for every prime ideal $\mathbf{p}$ there exists $-\infty\leq \ell\leq k<\infty$ such that $B=\mathbf{p}^kR_\mathbf{p}$ and $C=\mathbf{p}^\ell R_\mathbf{p}$. It is easy to see that $\varphi_\mathbf{p}$ is an epimorphism.
\end{proof}

The following result is known as Baer's Lemma. For various proofs we refer to \cite[Theorem 4.5]{Al-hon} (the endomorphism ring of a rank 1 torsion-free $R$-module is Dedekind by \cite[Lemma 1.1]{clab}) or \cite[Theorem 4.16]{lady}.

\begin{lemma}\label{baer}
Suppose that $A$ is a torsion-free $R$-module of finite rank. Then every short exact sequence  $$0\to C\to A\to B\to 0$$ of $R$-modules such that 
\begin{enumerate}[{\rm (i)}]
 \item 
$B$ is torsion-free of rank $1$, and 
\item $A=C+\sum_{f\in \Hom(B,A)}f(B)$ \end{enumerate} splits.   
\end{lemma}

As a corollary, we obtain the version of \cite[Theorem 2.3]{Arnold} for Dedekind domains.

\begin{corollary}\label{cor-baer}
Let $A\doteq B\oplus C$ be a quasi-decomposition of the finite rank torsion-free $R$-module $A$ such that 
$B$ is of rank $1$ and $\mathbf{type}(B)\leq \mathbf{IT}(C)$. Then $A$ has a direct summand $B'$ which is quasi-isomorphic to $B$.
\end{corollary}

\begin{proof}
Let $C_\star$ be the pure submodule generated by $C$. Then $A/C_\star$ is quasi-isomorphic to $B$, and it follows that the exact sequence $$0\to C_\star\to A\to  A/C_\star\to 0$$  verifies the hypotheses of Lemma \ref{baer}. Then $C_\star$ is a direct summand of $A$ and a complement of $C_\star$ is quasi-isomorphic to $B$.
\end{proof}

\begin{theorem}\label{tffr-main}
The following are equivalent for a finite rank torsion-free $R$-module $A$ over a Dedekind domain $R$:
\begin{enumerate}[{\rm (a)}]
 \item $A$ is a self-pure-generator;
 \item all pure rank $1$ submodules of $A$ are $A$-generated; 
 \item $A$ has a quasi-direct summand $B$ of rank $1$ such that $\mathbf{type}(B)=\mathbf{IT}(A)$;
 \item $A$ has a direct summand $B$ of rank $1$ such that $\mathbf{type}(B)=\mathbf{IT}(A)$. 
\end{enumerate}
\end{theorem}

\begin{proof} (a)$\Rightarrow$(b) is obvious.

(b)$\Rightarrow$(c) We consider a quasi-decomposition $A_1\oplus \dots\oplus  A_\ell$ for the module $A$ such that $A_1\oplus \dots\oplus  A_\ell\leq A$ and all submodules $A_i$ are strongly indecomposable. 

For all quasi-isomorphism classes associated to the $R$-modules $A_1, \dots,  A_\ell$ we fix representatives $B_u\in\{A_1,\dots,A_\ell\}$ with $u=1,\dots,n$ for some appropriate $n$.   

\begin{claim}\label{claim1}
For every non-zero element $x\in A$ there exists $u\in\{1,\dots,n\}$ such that $\Hom(B_u,\langle x\rangle_\star)\neq 0$. 
\end{claim}
In order to prove this, let $x\in A$ be a non-zero element of $A$. Since $A$ is a self-pure-generator, there exists a non-zero endomorphism $f:A\to A$ with  $f(A)\subseteq \langle x\rangle_\star$. Then there exists $i\in\{1,\dots,\ell\}$ such that $f(A_i)\neq 0$. Since $A_i$ is quasi-isomorphic to some $B_u$, it follows that $B_u$ is isomorphic to an essential submodule of $A_i$, hence $\Hom(B_u,\langle x\rangle_\star)\neq 0$, and the claim is proved.  

\begin{claim}\label{claim2} For every non-zero element $a\in A$ there exists a quasi-direct summand $B_u$ of rank $1$, and a non-zero homomorphism $f:B_u\to \langle a\rangle_\star$. 
\end{claim}

We start with a non-zero element $a=a_0\in A$. By Claim \ref{claim1} there exists $u_1\in \{1,\dots,n\}$ and a non-zero homomorphism 
$f_1:B_{u_1}\to \langle a\rangle_\star$. We fix an element $a_1\in B_{u_1}$ such that $f_1(a_1)\neq 0$. Therefore, we can construct inductively a sequence 
$$a_0,a_1,\dots,a_k,\dots$$ of elements in $A$ with the following properties: 
\begin{itemize}
\item[i)] for every $k\geq 1$ there exists $u_k\in \{1,\dots,n\}$ with $a_k\in B_{u_k}$, and 
\item [ii)] for every $k\geq 1$ there exists a homomorphism $f_k:B_{u_k}\to \langle a_{k-1}\rangle_\star$ such that $f_k(a_k)\neq 0$. 
\end{itemize}

We extend all homomorphisms $f_k$ to endomorphisms $$\overline{f_k}:A_1\oplus \dots\oplus  A_\ell\to A_1\oplus \dots\oplus  A_\ell$$ such that $\overline{f_k}(A_i)=0$ for all direct summands $A_i\neq B_{u_k}$. Note that all homomorphisms $\overline{f_k}$ can be viewed as quasi-endomorphisms of $A$, i.e. as restrictions of some endomorphisms of the $\Q$ vector space $\Q A$. In this context, it follows that for every $k\geq 1$ there exists a non-zero element $\alpha_k\in \Q$ such that $\overline{f_k}(a_k)=\alpha_ka_{k-1}$. It follows that for all $k\geq 0$ we have $\overline{f_1}\circ\overline{f_2}\circ\ldots \circ\overline{f_k}\neq 0$. 

Since the Jacobson radical of the quasi-endomorphism ring of $A$ is nilpotent, it follows that there exists $k\in \N$ such that $\overline{f_k}$ is not in the Jacobson radical of $\Q\End(A)$. Using the proof of \cite[Theorem 9.10]{Arnold} (see also the proof of \cite[Theorem 3.2]{Al-Go-pams}), we obtain that $B_{u_k}=B_{u_{k+1}}$, and $f_k\notin N(\End(B_{u_k}))$. Since $B_{u_{k}}$ is strongly indecomposable, it follows that $f_k$ is a quasi-isomorphism. But the image of $f_k$ is of rank 1, hence $B_{u_k}$ is of rank $1$, and Claim \ref{claim2} is proved.  

\medskip

In order to complete the proof, since $A$ has quasi-direct summands of rank $1$, we can suppose for the quasi-decomposition $ A_1\oplus \dots\oplus  A_\ell$ of $A$ that there exists $p\in\{1,\dots,\ell\}$ such that $A_i$ are of rank $1$ for all $i\in\{1,\dots,p\}$, and for all $i\notin\{1,\dots,p\}$ the $R$-modules $A_i$ are of rank at least $2$. Since the quasi-direct decomposition $A_1\oplus \dots\oplus  A_\ell$ of $A$ is unique up to quasi-isomorphism, and the types of rank $1$ torsion-free $R$-modules are invariant with respect to quasi-isomorphisms, it follows from Claim \ref{claim2} that for every $a\in A$ there exists $i\in\{1,\dots,p\}$ such that $\mathbf{type}(A_i)\leq \mathbf{type}(\langle a\rangle_\star)$. 

For every $i\in \{1,\dots,p\}$ we fix a non-zero element $a_i\in A_i$. Then there exists a quasi-direct summand $B$ of rank $1$ and a non-zero homomorphism $B\to \langle a_1+\dots+a_p\rangle_\star$. Since $\mathbf{type}(\langle a_1+\dots+a_p\rangle_\star)\leq \mathbf{type}(A_i)$ for all $i\in \{1,\dots,p\}$, it follows that there exists a quasi-direct summand $B$ of $A$ such that for all $a\in A$ we have $\mathbf{type}(B)\leq \mathbf{type}(\langle a\rangle_\star)$. It is easy to see that $\mathbf{type}(B)\leq \mathbf{IT}(A)$, hence the proof is complete.

\medskip

(c)$\Rightarrow$(d) Suppose that $B'\oplus C$ is a submodule of $A$ such that it is quasi-equal to $A$, $B'$ is of rank $1$, and $\mathbf{type}(B')\leq \mathbf{IT}(A)$. By Corollary \ref{cor-baer} it follows that $A$ has a direct summand $B$ of rank $1$ such that $$\mathbf{type}(B)=\mathbf{type}(B')\leq \mathbf{IT}(A).$$

\medskip

(d)$\Rightarrow$(a) From (d) it follows that for every pure submodule $C$ of $A$ we have $\mathbf{type}(B)\leq \mathbf{IT}(C)$. Now the conclusion follows from \cite[Proposition 3.7]{lady} and Lemma \ref{lemma1}.
\end{proof}

It was proved in \cite[Example 1.2]{Zimm} that there exist self-generator modules $M$ such that $M^2$ is not a self-generator. Moreover, over Dedekind domains all finite powers of a self-generator are also self-generators, \cite[Theorem 2.1]{Zimm}. A similar property is valid for finite rank torsion-free self-pure-generators over Dedekind domains.

\begin{corollary}
The following are equivalent for a finite rank torsion-free $R$-module $A$ over a Dedekind domain:
\begin{enumerate}[{\rm (a)}]
 \item $A$ is a self-pure-generator;
 \item if  $n$ is a positive integer then $A^n$ is a self-pure-generator;
\item there exists a positive integer $n$ such that $A^n$ is a self-pure-generator.
 \end{enumerate}
\end{corollary}

\begin{proof}
(a)$\Rightarrow$(b) Since $\mathbf{IT}(A)=\mathbf{IT}(A^n)$ for all $n$, we can apply Theorem \ref{tffr-main} to obtain the conclusion.

The implication (b)$\Rightarrow$(c) is obvious.

(c)$\Rightarrow$(a) Using Theorem \ref{tffr-main} we observe that $A^n$ has a rank $1$ quasi-direct summand $B$ such that $\mathbf{type}(B)\leq  \mathbf{IT}(A^n)=\mathbf{IT}(A)$. If we consider a quasi-decomposition of $A$ into a direct sum of strongly indecomposable $R$-modules, it induces a quasi-decomposition of $A^n$ into a direct sum of strongly indecomposable $R$-modules. Since these quasi-decompositions are unique up to quasi-isomorphisms, it follows that $B$ is quasi-isomorphic to a quasi-direct summand of $A$, and the proof is complete. 
\end{proof}

\section{Direct products of abelian groups}

In \cite{cal-et-al} it is proved that a direct product $G=\prod_{i\in I}G_i$ of a family of rank $1$ torsion-free groups is a self-pure-generator if and only if for every $x\in G$ there exists $i\in I$ such that $\type(G_i)\leq \type(x)$. In the case when all $G_i$ are of the same type $\tau$, $G$ is a self-pure-generator if and only if $\tau$ is idempotent. In the following we apply Theorem \ref{tffr-main} to characterize self-pure-generators which are direct products of finite rank torsion-free groups. In order to do this we have to recall some facts about types of rank $1$ torsion-free groups, and to fix some notations. 

Let $\bbP$ be the set of all (positive) prime integers. A family $\chi=(\chi_p)_{p\in\bbP}$ such that $\chi_p\in\N\cup\{\infty\}$ for all $p\in \bbP$ is called a \textsl{characteristic}. For a fixed characteristic $\chi$ we define the following sets of primes: 
\begin{align*}
S_\infty(\chi)&= \{p\in \bbP\mid \chi_p=\infty\},\\
S_f(\chi)&= \{p\in \bbP\mid 0<\chi_p<\infty\}, \textrm{ and }\\
S_0(\chi)&= \{p\in \bbP\mid \chi_p=0\}.
  \end{align*}

Two characteristics $\chi$ and $\vartheta$ are equivalent if $S_\infty(\chi)=S_\infty(\vartheta)$, and  $\chi_p=\vartheta_p$ for almost all $p\in \bbP\setminus S_\infty(\chi)$. The equivalence classes induced by this equivalence relation are called \textrm{types}. If $G$ is a torsion-free group and $g\in G$ then for every $p\in \bbP$ we define the $p$-height of $g$ as 
$$h_p(g)=\begin{cases}
          \infty, \textrm{ if } g\in p^kG \textrm{ for all } k\in\N;\\
          \textrm{the bigest non-negative integer } k_p \textrm{ such that } g\in p^{k_p}G, \textrm{ otherwise}.
         \end{cases}
$$
The family $\chi(g)=(h_p(a))_{p\in\bbP}$ is called the \textsl{characteristic of $g$}. The type represented by this characteristic is called  \textsl{the type of $g$}, and it is denoted by $\type(g)$. If $h\in \langle g\rangle_\star$ is a non-zero element then $\type(g)=\type(h)$. Therefore, all non-zero elements of a rank $1$ torsion-free group $G$ have the same type, and this is called \textsl{the type of $G$}. 

Two rank $1$ torsion-free groups are (quasi-)isomorphic if and only if they have the same type, hence this definition for the notion of type fits with the the notion of type defined for modules over Dedekind domains. We recall that if $\tau$ and $\upsilon$ are types represented by the characteristics $\chi=(\chi_p)_{p\in\bbP}$ and $\vartheta=(\vartheta_p)_{p\in\bbP}$ then $\tau\leq \upsilon$ if and only if $\chi_p\leq \vartheta_p$ for almost all $p$, and if $\chi_p\nleq \theta_p$ then $\chi_p$ is finite.     

In order to study direct products of finite rank torsion-free groups, it is enough to consider direct products with all factors indecomposable. In this context, we recall from \cite{Mad-sch} that every finite rank torsion free group $A$ has a direct decomposition $A=B\oplus C$, where $B=0$ or $B$ is a direct sum of torsion-free groups of rank $1$, and $C$ is a torsion-free group which has no direct summands of rank $1$. Moreover, $C$ is unique up to a near-isomorphism. 

Recall that a subseteq $J\subseteq I$ is \textsl{cofinite} if $I\setminus J$ is finite.

\begin{theorem}\label{direct-products}
Let $I$ be a set whose cardinality is smaller than the first measurable cardinal. If $G_{i}$, $i\in I$, is a family of indecomposable finite rank torsion-free groups, the following are equivalent:
\begin{enumerate}[{\rm (1)}]
 \item $G=\prod_{i\in I}G_i$ is a self-pure-generator;
 \item \begin{enumerate}[{\rm (a)}]
        \item there exists $k_0\in I$ such that $G_{k_0}$ is of rank $1$ and $$\type(G_{k_0})\leq \mathbf{IT}(G_i)$$ for all $i\in I$, and
        \item if $\chi$ is a (fixed) representative characteristic for $\type(G_{k_0})$, there exist a cofinite subset 
        $S\subseteq S_f(\chi)$ and a cofinite subset $J\subseteq I$ such that for all $p\in S$ and all $j\in J$ the the groups $G_j$ are $p$-divisible.
       \end{enumerate}
 \item there exist $k_0\in I$ and a cofinite subset $J\subseteq I$ such that 
 \begin{enumerate}[{\rm (a)}]
 \item $G_{k_0}$ is of rank $1$,
 \item if $\type(G_{k_0})$ is represented by the characteristic $\chi$ then $\prod_{j\in J}G_j$ is $p$-divisible for almost all $p\in \bbP\setminus S_0(\chi)$, and 
 \item $\type(G_{k_0})\leq \mathbf{IT}(G_i)$ for all $i\in I\setminus J$.  
 \end{enumerate}
\end{enumerate}
\end{theorem}

\begin{proof}
(1)$\Rightarrow$(2)
Let $G_i$, $i\in I$, be a family of indecomposable finite rank torsion-free groups such that $G=\prod_{i\in I}G_i$ is a self-pure-generator. For every $j\in I$, we denote by $\upsilon_j:G_j\to G$ the canonical embedding, and by $\pi_j:G\to G_j$ the canonical projection.  

Suppose that all $G_i$ are of rank at least $2$. Then for every $i\in I$ there exists $g_i\in G_i$ such that $\Hom(G_i,\langle g_i\rangle_\star)=0$. Since $G$ is a self-pure-generator, there exists a non-zero homomorphism $\alpha:G\to \langle (g_i)_{i\in I}\rangle_\star$. Since the pure subgroup $\langle (g_i)_{i\in I}\rangle_\star$ is a reduced torsion-free group of rank $1$, it is slender, hence there exists an index $j\in I$ such that   $\alpha \upsilon_j:G_j\to \langle (g_i)_{i\in I}\rangle_\star$ is a non-zero homomorphism.  Then $\pi_j\alpha\upsilon_j(G_j)$ is a non-zero subgroup of $\langle g_j\rangle_\star$, and this contradicts the initial assumption. Therefore, the set 
$$K=\{k\in I\mid G_k \textrm{ is of rank } 1\}$$ is nonempty.

Let $J=I\setminus K$. For every $i\in I$ we consider an element $(g_i)_{i\in I}$ such that $g_i\neq 0$ for all $i\in I$, and $\Hom(G_j,\langle g_j\rangle_\star)=0$ for all $j\in J$. Using a similar technique as above, it follows that there exist $k_0\in K$ and non-zero homomorphisms $\alpha_i:G_{k_0}\to \langle g_i\rangle_\star$ for all $i\in I$. This implies that there exists $k_0\in K$ such that $\type(G_{k_0})\leq \type(G_k)$ for all $k\in K$. Moreover, it follows that for every $j\in J$ and every element $g_j\in G_j$ such that $\Hom(G_j,\langle g_j\rangle_\star)=0$ we have $\type(G_{k_0})\leq \type(g_j)$. 

Suppose that there exists $j_0\in J$ and an element $h\in G_{j_0}$ such that $\type(G_{k_0})\nleq \type(h)$. We consider an element $g=(g_i)_{i\in I}$ such that $g_{k_0}\neq 0$, $g_{j_0}=h$, and for all $j\in J\setminus \{j_0\}$ we have $g_j\neq 0$ and $\Hom(G_j,\langle g_j\rangle_\star)=0$. Since $\Hom(G,\langle g\rangle_\star)\neq 0$, it follows that there exist $u\in I$ and a non-zero homomorphism $G_u\to \langle g\rangle_\star$. As above, it follows  that there exist a pure subgroup $L\leq G_u$ such that $G_u/H$ is of rank $1$ and there are non-zero homomorphisms $\alpha:G_u/L\to G_{k_0}$, $\beta:G_u/L\to \langle h\rangle_\star$, and $\alpha_j:G_u/L\to \langle g_j\rangle_\star$ for all $j\in J\setminus\{j_0\}$. By the choice of the elements $g_j$, $j\in J$, we observe that $u\notin J\setminus \{j_0\}$. If $u\in K$,  it follows that $\type(G_{k_0})\leq \type(G_u)\leq \type(h)$, a contradiction. Then $u=j_0$.  
Therefore, we proved that if $G_{j_0}$ has an element $h\in G_{j_0}$ such that $\type(G_{k_0})\nleq \type(h)$ then there exists $L\leq G_{j_0}$ such that $G_{j_{0}}/L$ is torsion-free of rank $1$, $\type(G_{j_{0}}/L)\leq \type(G_{k_0})$ and $\type(G/L)\leq \type(h)$. It follows that all rank $1$ pure subgroups of $G_{j_0}$ are generated by $G_{j_0}$, hence $G_{j_0}$ is a self-pure-generator. Since $G_{j_0}$ is indecomposable of rank at least $2$, we obtain a contradiction by using Theorem \ref{tffr-main}, hence the proof for (a) is complete.  

Suppose that (b) is not true.  Therefore, $S_f(\chi)$ and $I$ are infinite, and for every cofinite subset $S\subseteq S_f(\chi)$ and every cofinite subset $J\subseteq I$ there exist $p\in S$ and $j\in J$ such that $G_j$ is not $p$-divisible. 

Applying this for $S_1=S_f(\chi)$ and $J_1=I\setminus\{k_0\}$, it follows that we can find $p_1\in S_1$ and 
$j_1\in J_1$ such that $G_{j_1}$ is not $p_{j_1}$-divisible. Now, suppose that we constructed $k$ pairwise different primes $p_1,\dots,p_k\in S_{f}(\chi)$ and $k$ pairwise different $i_1,\dots,i_k\in I$ such that for all $\ell=1,\dots,k$ the group $G_{i_\ell}$ is not $p_{\ell}$-divisible. Using again our assumption for $S=S_f(\chi)\setminus\{p_1,\dots,p_k\}$ and $J=I\setminus\{i_1,\dots,i_k\}$, it follows that there exist $p_{k+1}\in S_f(\chi)\setminus\{p_1,\dots,p_k\} $ and $j_{k+1}\in I\setminus\{i_1,\dots,i_k\}$ such that $G_{j_{k+1}}$ is not $p_{k+1}$-divisible. Inductively, we conclude that there exist an infinite sequence $p_1,\dots,p_k,\dots$ of pairwise different primes from $S_f(\chi)$ and an infinite sequence $i_1,\dots,i_k,\dots$ of pairwise different elements of $I$ such that for every positive integer $k$ the group $G_{i_{k}}$ is not $p_k$-divisible. 

Since the groups $G_{i_{k}}$ are not $p_k$-divisible, for every positive integer $k$ there exists an element $g_{i_k}\in G_{i_{k}}$ such that $h_{p_{k}}(g_{i_{k}})=0$. We take an element $g=(g_i)_{i\in I}\in G$ such that $g_i=g_{i_k}$ for all $i\in\{i_1,\dots,i_k,\dots\}$. If $\vartheta$ is the characteristic of $g$, it follows that $S_f(\chi)\subseteq S_0(\vartheta)$. Since $S_f(\chi)$ is infinite, it follows that $\type(G_{k_0})\nleq \type(\langle g\rangle_\star)$. Moreover, $\type(G_{k_0})\leq \mathbf{IT}(G_i)$ for all $i\in I$, hence $\Hom(G_i,\langle g\rangle_\star)=0$ for all $i\in I$. By using \cite[Corollary 94.5]{Fuchs}, we obtain 
$\Hom(G, \langle g\rangle_\star)=0$, a contradiction, and the proof for (b) is complete.

(2)$\Rightarrow$(3) Since a direct product of abelian groups is $p$-divisible if and only if all its factors are $p$-divisible, it is easy to see that the group $G_{k_0}$ and the set $J$ exhibited in (2) have the required properties.

(3)$\Rightarrow$(1) We write $G=G_{k_0}\oplus H\oplus L$, where $H=\prod_{j\in J}G_j$, and $L=\prod_{i\in I\setminus (J\cup\{k_0\})}G_i$. From (b) it follows that $\type(G_{k_{0}})\leq \type(x)$ for all $x\in H$, and (c) implies that 
$\type(G_{k_{0}})\leq \type(y)$ for all $y\in L$. It follows, as in the proof of Lemma \ref{lemma1} that all pure subgroups of $G$ are $G_{k_0}$-generated, and the proof is complete. 
\end{proof}

In the following example we present a self-pure generator of infinite rank which has no rank $1$ direct summands.

\begin{example}\label{example}
There exists a groups $G$ which is a direct sum of a countable family of rank $2$ torsion-free abelian groups such that $G$ is a self-pure-generator, but it has no direct summands of rank $1$. 
\end{example}

\begin{proof}
We index the set $\bbP$ of all prime integers as $\bbP=\{p_n\mid n\in\Z\}$, and for every integer $k$ we consider the (idempotent) type $\tau_k$ represented by a characteristic $\chi=(\chi_{k,n})$, where 
$$\chi_{k,n}=\begin{cases}                                                                                                                                                          \infty \textrm{ if } n<k, \textrm{ and }\\ 
0 \textrm{ if } n\geq k.                                                                                                                                                         \end{cases}
$$
Therefore, $\tau_k<\tau_{k+1}$ for all $k\in \Z$. Since the types $\tau_k$ are idempotent, it follows that for  every $k\in \Z$ there exists a (unital) subring $R_k$ of $\bbQ$ such that $\mathrm{type}(R_k)=\tau_k$. 

We claim that for every integer $k$ there exists a strongly indecomposable torsion free abelian group $G_k$ of rank $2$ such that 
\begin{itemize}
 \item[(i)] 
for every $g\in G_k$ we have $\langle g\rangle_\star\cong R_k$, and 
\item [(ii)] for every $k\in\mathbb{Z}$ there exists an exact sequence of the form   $$0\to R_k\to G_k\to  R_{k+1}\to 0.$$ 
\end{itemize}
In order to prove this, we will use the construction used in \cite[Section 1]{Fuchs-Lostra}. Note that $R_k$ is a subgroup of $R_{k+1}$ and
$R_{k+1}/R_k\cong \Z(p_{k+1}^\infty)$. For all $k\in \Z$ we will identify $R_{k+1}/R_k$ with $\Z(p_{k+1}^\infty)$. Let $\varphi:R_{k+1}\to \Z(p_{k+1}^\infty)$ be the homomorphism obtained by multiplying the canonical surjection $R_{k+1}\to R_{k+1}/R_k$ by a fixed transcedental $p_{k+1}$-adic integer (hence $\varphi$ cannot be lifted at an endomorphism of $R_{k+1}$). Let us remark that $\Ker(\varphi)\cong R_k$. We construct a pullback diagram 
$$\xymatrix{ & & R_k\ar@{=}[r]\ar@{>->}[d]^{\beta} & R_k\ar@{>->}[d] \\
0\ar[r]& R_k\ar@{=}[d]\ar[r]^{\alpha} & G_k \ar[r] \ar@{->>}[d] & R_{k+1}\ar[r]\ar@{->>}[d]^\varphi & 0\\
0\ar[r]& R_k\ar[r]^{\iota} & R_{k+1} \ar[r]^{\pi}  & \Z(p^\infty_{k+1})\ar[r] & 0,}
$$
where $\iota$ is the canonical injection and $\pi$ is the canonical surjection. 

We will prove that $G_k$ is homogeneous of type $\tau_k$. We observe that $\alpha(R_{k})$ and $\beta(R_k)$ are pure rank $1$ subgroups of $G_k$. By the choice of $\varphi$, it follows that $\alpha(R_{k})\cap \beta(R_k)=0$. Hence $G_k$ is a group of rank $2$ which can be embedded in $R_{k+1}\times R_{k+1}$.  It follows that the type of every non-zero element of $G_k$ is $\tau_k$ or $\tau_{k+1}$. If there exists an element of type $\tau_{k+1}$ in $G_k$ then $\alpha$ splits by Baer's Lemma. But this is not possible since $\varphi$ cannot be lifted to an endomorphism of $R_{k+1}$. Therefore, $G_k$ is homogeneous of type $\tau_k$, and using \cite[Theorem 3.2]{Arnold} it follows that $G_k$ is strongly indecomposable. 

Let $G=\bigoplus_{k\in\Z}G_k$. For every non-zero element $g\in G$ there exists $k\in \Z$ such that $\type(g)=\tau_k$, and it follows that there exists an epimorphism $G_{k-1}\to \langle g\rangle_\star$. It follows that $G$ is a self-pure generator. 

Now, suppose that $G$ has a rank $1$ direct summand $A$. Then we can find an element $a\in G$ such that $A=\langle a\rangle_\star$. There exist $m\in\Z$ such that $\type(a)=\tau_m$. It follows that the $m$-th component of $a$ is non-zero
and there exists $n\in\Z$ with $m\leq n$ such that $a\in \bigoplus_{k=m}^n G_k$. From $A\leq \bigoplus_{k=m}^n G_k$, it follows that $A$ is a direct summand of $\bigoplus_{k=m}^n G_k$. Since $G_m$ is indecomposable, we have $m<n$. 
But for every $k>m$ the group $G_k$ is homogeneous of type $\tau_k>\tau_m$, and it follows that $\Hom(\bigoplus_{k=m+1}^n G_k, A)=0$. Using the fact that $A$ is a direct summand, it follows that there exists a (non-zero) projection $\pi:\bigoplus_{k=m}^n G_k\to A$, and we obtain that $\Hom(G_m,A)\neq 0$. Since $A$ is isomorphic to a rank $1$ (pure) subgroup of $G_m$, it follows that $G_m$ has an endomorphism which is not a quasi-isomorphism. But this is not possible by \cite[Theorem 3.3]{Arnold}. Then our assuption is false, hence $G$ has no direct summands of rank $1$.     
\end{proof}

 \end{document}